\documentclass[11pt]{amsart}

\usepackage[utf8]{inputenc}
\usepackage[T1]{fontenc}
\usepackage[english]{babel}
\usepackage{amsmath}
\usepackage{amsthm}
\usepackage{amssymb}
\usepackage{mathrsfs}
\usepackage{dsfont}
\usepackage{array}
\usepackage{graphicx} 
\usepackage[left=3cm,right=3cm,top=3cm,bottom=3cm]{geometry}
\usepackage{hyperref}
\usepackage{tikz-cd}
\usetikzlibrary{shapes}
\usetikzlibrary[patterns]
\usetikzlibrary{decorations.pathreplacing}
\usepackage{caption}
\usepackage{subcaption}
 \usepackage{comment}
\usepackage{enumitem}
\setitemize{label=$\triangleright$}

\usepackage{eucal} 

\usepackage{mlmodern}

\newcommand{\bcd}{\begin{center}\begin{tikzcd}}
\newcommand{\ecd}{\end{tikzcd}\end{center}}

\newcommand{\ev}{\mathrm{ev}}
\newcommand{\pt}{\mathrm{pt}}

\newcommand{\PP}{\mathbb{P}}

\newcommand{\ZZ}{\mathbb{Z}}
\newcommand{\RR}{\mathbb{R}}
\newcommand{\CC}{\mathbb{C}}

\newcommand{\QQ}{\mathbb{Q}}

\newcommand{\TT}{\mathbb{T}}

\newcommand{\AAA}{\mathscr{A}}

\newcommand{\CCC}{\mathscr{C}}

\renewcommand{\P}{\mathcal{P}}
\renewcommand{\L}{\mathcal{L}}

\newcommand{\M}{\mathcal{M}}

\newcommand{\gen}[1]{\langle #1 \rangle}

\newtheorem{theo}{Theorem}[section]
\newtheorem*{theom}{Theorem}
\newtheorem{prop}[theo]{Proposition}

\newtheorem{lem}[theo]{Lemma}

\theoremstyle{definition}
\newtheorem{defi}[theo]{Definition}

\theoremstyle{remark}
\newtheorem{remark}[theo]{Remark}
\newenvironment{rem}[1]{
    \begin{remark}#1}{
    \xqed{\blacklozenge}\end{remark}
}

\theoremstyle{remark}
\newtheorem{example}[theo]{Example}
\newenvironment{expl}[1]{
    \begin{example}#1}{
    \xqed{\lozenge}\end{example}
}

\newcommand{\xqed}[1]{
    \leavevmode\unskip\penalty9999 \hbox{}\nobreak\hfill
    \quad\hbox{\ensuremath{#1}}}

\def\floor (#1) at (#2,#3); {
    \node[draw,ellipse, minimum width=1cm, minimum height = 0.6 cm] (#1) at (#2,#3) {$\bullet$} ;
}

\def\ufloor (#1) at (#2,#3) (#4); {
    \node[draw,ellipse, minimum width=1cm, minimum height = 0.6 cm] (#1) at (#2,#3) {\scriptsize #4} ;
}

\def\marked (#1) to (#2) pos=#3 in=#4 out=#5; {
   \draw (#1) to[out=#5,in=#4] node[pos=#3] {$\bullet$} (#2) ;
}

\def\leftedge (#1) to (#2) pos=#3 in=#4 out=#5 w=#6; {
   \draw (#1) to[out=#5,in=#4] node[midway,left] {$#6$} (#2) ;
}
\def\leftmarked (#1) to (#2) pos=#3 in=#4 out=#5 w=#6; {
   \draw (#1) to[out=#5,in=#4] node[pos=#3] {$\bullet$} node[midway,left] {$#6$} (#2) ;
}

\def\wlmarked (#1) to (#2) pos=#3 in=#4 out=#5 w=#6; {
   \draw (#1) to[out=#5,in=#4] node[pos=#3] {$\bullet$} node[midway,left] {$#6$} (#2) ;
}

\def\rightedge (#1) to (#2) pos=#3 in=#4 out=#5 w=#6; {
   \draw (#1) to[out=#5,in=#4] node[midway,right] {$#6$} (#2) ;
}
\def\rightmarked (#1) to (#2) pos=#3 in=#4 out=#5 w=#6; {
   \draw (#1) to[out=#5,in=#4] node[pos=#3] {$\bullet$} node[midway,right] {$#6$} (#2) ;
}

\def\doublemarked (#1) to (#2) pos=#3 in=#4 out=#5; {
   \draw (#1) to[out=#5,in=#4] node[pos=#3] {$\bullet$} (#2) ;
}

\usepackage{todonotes}

\makeatletter
\def\l@subsection{\@tocline{2}{0pt}{2.5pc}{5pc}{}}
\makeatother

\makeatletter
\renewcommand{\l@section}{\@tocline{1}{0pt}{10pt}{1pc}{\bfseries}}
\makeatother
\title{A short proof of the multiple cover formula for point insertions} 
\author{Thomas Blomme}

\address{Université de Neuchâtel, rue \'Emile Argan 11, Neuchâtel 2000, Suisse}
\email{thomas.blomme@unine.ch}


\begin{document}

\maketitle

\begin{abstract}
A few years ago, G. Oberdieck conjectured a multiple cover fomula that determines the number of curves of fixed genus and degree passing through a configuration of points in an abelian surface. This formula was proved by the author using tropical techniques and Nishinou's correspondence theorem. Using the same techniques, we give a much shorter proof of the multiple cover formula for point insertions, relying on the same geometrical idea, but avoiding any kind of tropical enumeration.
\end{abstract}

\tableofcontents


\section{Introduction}

\subsection{Setting and enumerative problem} Compact complex surfaces with trivial canonical bundle are either K3 surfaces, already at the center of many studies, or abelian surfaces, for which there also exists a vast literature, though they seem to be less studied in the realm of enumerative geometry.

As a complex variety, an abelian variety $\CC A$ is a complex torus, \textit{i.e.} the quotient of $\CC^2$ by a rank $4$ lattice $L$. Through the exponential map, it may also be seen as a quotient of $(\CC^*)^2$ by a rank $2$ lattice $\Lambda$. To become an abelian variety, the complex torus needs to be endowed with the choice of a \textit{polarization} \cite[Chapter 2.6]{griffiths2014principles}, which is a skew-symmetric form $Q\in\wedge^2 L^*$ satisfying conditions known as \textit{Riemann-bilinear relations} (see Definition\ref{defi-riemann-bilin}), or alternatively an ample line bundle, of which $Q$ is actually the chern class.

The most naive enumerative geometry in abelian surfaces consists in counting the number of genus $g$ curves in a given homology class $\beta\in H_2(\CC A)\simeq\wedge^2 L$ passing through $g$ points. This problem was addressed and solved in \cite{bryan1999generating} for \textit{primitive classes}, which are elements of the second homology group $\wedge^2 L$ that cannot be expressed as a multiple of a smaller class.

The above enumerative problem can be set in the reduced Gromov-Witten theory framework, where it is shown to be deformation invariant: it does not depend on the abelian surface and only depends on the curve class $\beta$ through its divisibility $d$ and its self-intersection $\beta^2=2d^2n$. We denote the invariant by $N_{g,d,n}$.

Though the primitive case has been known for some times (see \cite{bryan1999generating} and \cite{bryan2018curve} for computation of other reduced Gromov-Witten invariants), tackling the computation of invariants for \textit{divisible} classes turns out to be quite challenging. Only a handful of computations are known in this case.

\subsection{Tropical limit and result}

In 2020, T.~Nishinou proved a correspondence theorem \cite{nishinou2020realization} for curves passing through points in abelian surfaces, adapting to the abelian setting the tropical correspondence theorem for curves in toric surfaces proved independently by G.~Mikhalkin \cite{mikhalkin2005enumerative} and T.~Nishinou-B.~Siebert \cite{nishinou2006toric}. The correspondence theorem transforms the enumerative problem into a combinatorial problem dealing with some piecewise linear graphs called \textit{tropical curves} in a real torus $\RR^2/\Lambda$, here known as \textit{tropical abelian surface}.

The correspondence theorem theoretically allows the computation of the above mentioned invariants for any curve class, provided one is able to deal with the combinatorics of the tropical enumerative problem. Without any further assumption, this is a difficult problem. Assuming the tropical abelian surface and the point configuration are \textit{stretched}, this problem was solved by the author in a series of paper \cite{blomme2022abelian1,blomme2022abelian2,blomme2022abelian3}. The highlight comes in \cite{blomme2022abelian3} with the proof a \textit{multiple cover formula} satisfied by the invariants, which we now state.

\begin{theom}\textbf{\ref{theo-multiple-cover}}
    The invariants $N_{g,d,n}$ satisfy the multiple cover formula
    $$N_{g,d,n} = \sum_{k|d}k^{4g-3}N_{g,1,\left(\frac{d}{k}\right)^2n}.$$
\end{theom}

In particular, the above multiple cover formula reduces the computation for non-primitive classes to primitive classes, for which we have an explicit formula from \cite{bryan1999generating}, also possible to achieve via tropical methods and the correspondence theorem.

This multiple cover formula is actually a specific case of a more general formula of the same shape that should be satisfied by all reduced Gromov-Witten invariants of abelian surfaces, as conjectured by G.~Oberdieck in \cite{oberdieck2022gromov}: the $N$ should be replaced by some GW-invariant and only the exponent $4g-3$ of $k$ may change. In the case of point insertions, this formula is the one stated in Theorem \ref{theo-multiple-cover}, and was already proven by the author in \cite{blomme2022abelian3}.

\subsection{Idea of proof}

The proof presented in \cite{blomme2022abelian3} is quite technical and relies on the fact that it is possible to choose tropical abelian surfaces with non-equivalent polarizations for which the tropical curves are \textit{floor decomposed}, \textit{i.e.} enclosed in a discrete data called \textit{diagrams}. The curve count is recovered by painful computations. The key is that the same diagrams can be used to compute the invariants for both a primitive class and a non-primitive class. One finally observes that the multiple cover formula is already satisfied at the level of diagrams.

Algebraically, the proof from \cite{blomme2022abelian3} may be seen as follows: it is possible to construct two families of abelian surfaces with distinct polarizations that degenerate to a central fiber which is chain of $E\times\PP^1$ (with $E$ an elliptic curve) glued along their divisors $E\times\{0/\infty\}$. Though the pieces are the same, the central fibers differ by a twist (translation) in the $E$ direction when going around the chain of $\PP^1$. This twist is responsible for the distinct polarizations. A decomposition formula breaks the invariants as a sum of pieces indexed by \textit{diagrams}, and the multiple cover formula is already true at the level of diagrams.

In \cite{blomme2022abelian3}, the construction from the previous paragraph is achieved tropically by choosing a \textit{tropical twist}: the underlying tropical abelian surfaces are not the same though the diagrams are the same. The idea of this paper is to choose instead a \textit{complex twist}, so that the tropical abelian surfaces are this time the same. Therefore, the curves to which we apply the correspondence theorem are the same, and the multiple cover formula may now be seen at the level of the tropical curves. This way, the proof avoids any kind of tropical enumeration, solely relying of the existence of tropical correspondence, and no explicit solving.

The method of this paper could probably be generalized for other reduced Gromov-Witten invariants if one had a suitable decomposition formula, which is for now unknown. In particular, this method does not yet work to prove the multiple cover formula for refined invariants from \cite{blomme2022abelian3} since the latter are not related to anything on the complex setting yet.

\subsection{Plan of the paper}

In the second section we recall some basics about abelian surfaces, both complex and tropical. In the third section, we explain how a polarization on a Mumford family yields a tropical polarization of its tropicalization. In Section 4, we recall the setting of Nishinou's correspondence theorem \cite{nishinou2020realization} for Mumford families of abelian surfaces before proving the multiple cover formula in Section 5.

\textit{Acknowledgements.} The author would like to thank Francesca Carocci, for the 5 min conversation that lead to an epiphany and the surprising revelation that the same idea from \cite{blomme2022abelian3} actually yielded a much shorter proof.

\section{Abelian surfaces, curves and deformations}

    \subsection{Complex abelian surfaces}

    We start with the definition of a complex torus.

    \begin{defi}
        A complex torus is a quotient $\CC^2/L$, where $L$ is rank $4$ lattice in $\CC^2$.
    \end{defi}

    We denote a complex abelian surface by $\CC A$. We write $\Omega$ the $2\times 4$ \textit{period matrix} whose columns span $L$. Choosing a $\CC$-basis of $\CC^2$ inside $L$, we can write $L=\ZZ^2\oplus\Lambda\subset\CC^2$, where $\Lambda$ is a supplement of $\ZZ^2$. Via the exponential map coordinate by coordinate, we obtain a biholomorphic map
    $$\CC^2/L \to (\CC^*)^2/\Lambda,$$
    where the map $\Lambda\to (\CC^*)^2$ is the composition of the inclusion $\Lambda\to\CC^2$ and the coordinate exponential.

    \begin{rem}
        Composing with the logarithmic map $\log|-|:(\CC^*)^2\to\RR^2$, both lattices $\ZZ^2$ and $\Lambda$ may be seen in $\RR^2$: $\Lambda$ via the log map and $\ZZ^2$ via the canonical inclusion.
        
        An orientation of a lattice is a generator of its top-exterior power. We can take orientations of $\ZZ^2$ and $\Lambda$ that induce the same orientation of $\RR^2$, yielding a well-defined orientation of $L$. The orientation it induces on $\CC^2$ is however the opposite of the complex orientation. This is due to the fact that the two middle coordinates are switched.
    \end{rem}

    From now on, we assume that $L=\ZZ^2\oplus\Lambda$, and we have chosen compatible orientations of $\ZZ^2$ and $\Lambda$. The second homology and cohomology groups of $\CC A$ are respectively $\wedge^2 L$ and $\wedge^2 L^*$. Following \cite[Chapter 2.6]{griffiths2014principles}, we now define a polarization of a complex torus.
    
    \begin{defi}\label{defi-riemann-bilin}
        A polarization on $\CC A=\CC^2/L$ is an invertible skew-symmetric integer matrix $Q$ satisfying Riemann bilinear relations:
        $$\left\{ \begin{array}{l}
            \Omega Q^{-1}\Omega^\intercal =0, \\
            -i\Omega Q^{-1}\overline{\Omega}^\intercal  \text{ is (hermitian and) positive definite.}
        \end{array}\right.$$
    \end{defi}

    More intrinsically, $Q$ is an element of $\wedge^2L^*$, skew-symmetric forms on $L$, with skew-symmetric matrices after the choice of a basis of $L$ (\textit{i.e.} $\ZZ^2$ and $\Lambda$).

    Using the orientation of $L$ given by the choice of compatible orientations on $\ZZ^2$ and $\Lambda$, we have the pairing coming from the wedge-product:
    $$\wedge^2L^*\otimes\wedge^2L^*\longrightarrow \wedge^4L^*\simeq\ZZ.$$
    The associated quadratic form is actually twice the \textit{Pfaffian}.

    \begin{rem}
    As $H^2(\CC A)=\wedge^2L^*$, the integer matrix $Q$ is in fact the chern class of an ample line bundle $\L$ on $\CC A$. The zero-locus of a section of $\L$ is thus Poincar\'e dual to $Q\in H^2(\CC A)$. Due to the change of orientation, the intersection pairing on $H^2(\CC A)$ coming from the cup-product and the complex orientation is actually the opposite of the above Pfaffian pairing.
    \end{rem}

    \subsection{Poincaré duality and polarization}
    
    Before going to the tropical world, we give a formulation of Poincaré duality assuming the polarization takes a particular triangular form in the decomposition $\ZZ^2\oplus\Lambda$.
    
    \begin{lem}
        Assume $Q=\left(\begin{smallmatrix} 0 & C \\ -C^\intercal & T \end{smallmatrix}\right)$ for some integer matrices $C,T$. Then, we have the following:
        \begin{enumerate}
            \item The matrix $T$ is skew-symmetric of size $2$, its upper-right coefficient is equal to its Pfaffian, and $\operatorname{Pf}(Q)=-\det C$.
            \item If $B$ is the comatrix of $C$, we have
                $$\operatorname{Pf}(Q)Q^{-1} = \begin{pmatrix}
                -T & B \\
                -B^\intercal & 0
                \end{pmatrix}.$$
            \item The map $Q\in\wedge^2L^*\mapsto \operatorname{Pf}(Q)Q^{-1}\in\wedge^2L$ is actually the Poincar\'e duality map $H^2(\CC A)\to H_2(\CC A)$.
        \end{enumerate}
    \end{lem}

    \begin{proof}
    The first point follows from a direct computation. For the second point, the inverse matrix is given by
    $$Q^{-1} = \begin{pmatrix}
        (C^{-1})^\intercal TC^{-1} & -(C^{-1})^\intercal \\
        C^{-1} & 0
    \end{pmatrix}.$$
    We set $B=(\det C)(C^{-1})^\intercal$, which is actually the comatrix of $C$. In particular, it has integer coefficients. The Pfaffian of $(C^{-1})^\intercal TC^{-1}$, equal to its upper-right coefficient, is $\operatorname{Pf}(T)/\det C$, so that $(\det C)(C^{-1})^\intercal TC^{-1}=T$. Therefore, since $\operatorname{Pf}(Q)=-\det C$, we get the desired expression.
    
    The matrix $\operatorname{Pf}(Q)Q^{-1}$ is an element of $\wedge^2L=H_2(\CC A)$. It coincides with the Poincaré duality $H^2(\CC A)\to H_2(\CC A)$ by a direct computation on a basis of $\wedge^2L$ and $\wedge^2L^*$.
    \end{proof}
    
    Up to a change of basis of $L$, it is possible to assume that $T=0$ and $C$ is diagonal, so that $Q$ is in \textit{diagonal antidiagonal form}:
    $$Q=\left(\begin{smallmatrix}
    0 & 0 & d_1 & 0 \\
    0 & 0 & 0 & d_2 \\
    -d_1 & 0 & 0 & 0 \\
    0 & -d_2 & 0 & 0
    \end{smallmatrix}\right),$$
    where $d_1|d_2$. The numbers $d_1,d_2$ can actually be recovered as follows: $d_1$ is the divisibility of $Q$, \textit{i.e.} the gcd of the coefficients, and $d_1d_2=|\operatorname{Pf}(Q)|$. The pair $(d_1,d_2)$ is called the type of the polarization. However, finding such a basis may require to choose a different decomposition of $L=\ZZ^2\oplus\Lambda$, \textit{i.e.} the change of basis may not preserve the given decomposition.
    

    \subsection{Tropical abelian surfaces}

    We recall some basics on tropical abelian surfaces and tropical curves in the latter, following \cite{mikhalkin2008tropical}, \cite{blomme2022abelian1} and \cite{nishinou2020realization}.

    \begin{defi} Let $\Lambda$ be a rank $2$ lattice with an inclusion in $\RR^2$. Choosing a basis of $\Lambda$, the inclusion is given by a matrix $S$.
        \begin{itemize}
            \item The quotient $\TT A=\RR^2/\Lambda$ is called a tropical torus.
            \item A polarization on a tropical torus $\TT A$ is a linear map $C:\Lambda\to(\ZZ^2)^*$ (with $C$ as in chern) such that the bilinear map induced by the composition $S^\intercal C:\Lambda\to\Lambda^*_\RR$ is symmetric and positive definite.
            \item A tropical abelian variety is a tropical torus endowed with a choice of polarization.
        \end{itemize}
    \end{defi}

    \begin{rem}
        As $C:\Lambda\to(\ZZ^2)^*$ and $S:\Lambda\to\RR^2$, the composition $S^\intercal C$ does indeed make sense as a map $\Lambda\to\Lambda^*_\RR$, which induces a bilinear map $\Lambda\otimes\Lambda\to\RR$, which we may ask to be symmetric.
    \end{rem}

    Taking the adjoint, we have $C^\intercal:\ZZ^2\to\Lambda^*$, and $C$ can also be seen as an element of $\Lambda^*\otimes(\ZZ^2)^*$. The reason is that as in the classical case, the polarization is actually the (tropical) chern class of an ample tropical line bundle, \textit{i.e.} an element of $H^1(\TT A,(\ZZ^2)^*)=\Lambda^*\otimes(\ZZ^2)^*$.

\medskip

    An abstract (irreducible) tropical curve is (connected) metric graph $\Gamma$ with some unbounded edges called ends. The genus of an irreducible tropical curve is its first Betti number. The ends are used to encode marked points and avoid dealing with marked vertices. A tropical curve is trivalent if every vertex is adjacent to exactly three edges (or ends).

    \begin{defi} (parametrized tropical curve)
        \begin{itemize}
            \item A parametrized tropical curve is a map $h:\Gamma\to\TT A$ from an abstract tropical curve $\Gamma$, which is affine with integer slope on the edges, contracts the ends, and satisfies the balancing condition at vertices.
            \item The degree is the homology class realized by the curve in $H_{1,1}(\TT A)=H_1(\TT A,\ZZ^2)=\Lambda\otimes\ZZ^2$.
            \item The gcd $\delta_\Gamma$ of a parametrized tropical curve $h:\Gamma\to\TT A$ is the gcd of the lattice lengths of the slopes.
        \end{itemize}
    \end{defi}

    The degree of a tropical curve is denoted by $B$, as curve classes in the complex setting are usually denoted by $\beta$. The degree $B$ may be seen as a map $\Lambda^*\to\ZZ^2$ and written as a matrix after a choice of coordinates. Its transpose has the following explicit description: for $\varphi\in(\ZZ^2)^*$, we have
    $$B^\intercal(\varphi)=\sum_e \varphi(n_e)e\in H_1(\TT A,\ZZ)\simeq\Lambda,$$
    where the sum is over the edges of $\Gamma$, which are $1$-simplices in $\Lambda$, and $n_e$ is their slope. As the slope $n_e$ is reversed when the orientation is reversed, the sum does not depend on the orientation of the edges. The sum is indeed a cycle due to the balancing condition. Seen as an element of $\Lambda$, denoting by $l_e$ the length of the edge $e$, the isomorphism with $\Lambda$ is given by
    $$B^\intercal(\varphi) = \sum_e \varphi(n_e)l_e n_e \in\Lambda\subset\RR^2.$$

    \begin{lem}
        If $h:\Gamma\to\TT A$ is a tropical curve of degree $B$, then $\det B>0$ and $BS^\intercal$ is symmetric and positive definite.
    \end{lem}

    \begin{proof}
        Let $\varphi\in(\ZZ^2)^*$. We have an explicit expression of $B^\intercal(\varphi)$ as an element of $\Lambda\subset\RR^2$. Take a second linear form $\psi\in(\ZZ^2)^*$, with restriction $S^\intercal\psi$ on $\Lambda$. We get
        $$\psi(SB^\intercal(\varphi))=\sum_e l_e\varphi(n_e)\psi(n_e).$$
        As $l_e>0$, the symmetry and the positivity follows. As $\det S>0$, we deduce that $\det B>0$, which also follows from the positivity of the self-intersection of the tropical curve.
    \end{proof}
    
    After a choice of basis, both degree and polarization are represented by a $2\times 2$ matrix, related by the Poincaré duality. The analog of the transformation $Q\mapsto \operatorname{Pf}(Q)Q^{-1}$ is here given by $C\mapsto (\det C)(C^{-1})^\intercal$, which is actually the comatrix of $C$. It is also the restriction of $Q\mapsto\operatorname{Pf}(Q)Q^{-1}$ on antidiagonal matrices:
    \begin{itemize}
        \item If $C$ is a polarization, its comatrix $B=(\det C)(C^{-1})^\intercal$ is the degree of a curve.
        \item Conversely, given a degree class $B$, its comatrix $C=(\det B)(B^{-1})^\intercal$ is a polarization.
    \end{itemize}
    With $B$ and $C$ related as above, $S^\intercal C$ is symmetric if and only if $BS^\intercal$ symmetric.

    Following \cite{blomme2022abelian1}, it is usually easier to find the matrix $B$ than $C$, as the matrix of $B:\Lambda^*\to\ZZ^2$ may be read by choosing a fundamental domain of $\RR^2/\Lambda$, and making the sum of slopes of the edges intersecting the right and top sides of the fundamental domain respectively.

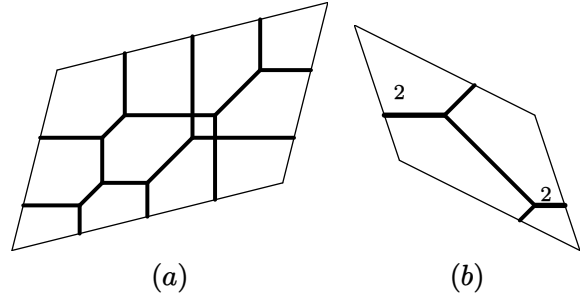
\begin{figure}
\begin{center}
\begin{tabular}{cc}
\begin{tikzpicture}[line cap=round,line join=round,x=0.3cm,y=0.3cm]
\clip(0,0) rectangle (14,11);
\draw [line width=0.5pt] (0,0)-- (12,3);
\draw [line width=0.5pt] (0,0)-- (2,8);
\draw [line width=0.5pt] (2,8)-- (14,11);
\draw [line width=0.5pt] (12,3)-- (14,11);

\draw [line width=1.5pt] (3,2)-- (4,3);
\draw [line width=1.5pt] (4,3)-- (4,5);
\draw [line width=1.5pt] (4,3)-- (6,3);
\draw [line width=1.5pt] (6,3)-- (8,5);
\draw [line width=1.5pt] (0.5,2)-- (3,2);
\draw [line width=1.5pt] (3,0.75)-- (3,2);
\draw [line width=1.5pt] (4,5)-- (5,6);
\draw [line width=1.5pt] (1.25,5)-- (4,5);
\draw [line width=1.5pt] (5,6)-- (5,8.75);
\draw [line width=1.5pt] (5,6)-- (9,6);
\draw [line width=1.5pt] (9,6)-- (11,8);
\draw [line width=1.5pt] (11,8)-- (13.25,8);
\draw [line width=1.5pt] (11,8)-- (11,10.25);
\draw [line width=1.5pt] (8,5)-- (8,9.5);
\draw [line width=1.5pt] (8,5)-- (12.5,5);
\draw [line width=1.5pt] (6,3)-- (6,1.5);
\draw [line width=1.5pt] (9,6)-- (9,2.25);

\begin{scriptsize}

\end{scriptsize}
\end{tikzpicture}
&
\begin{tikzpicture}[line cap=round,line join=round,x=0.3cm,y=0.3cm]
\clip(0,0) rectangle (10,10);
\draw [line width=0.5pt] (2,4)-- (0,10);
\draw [line width=0.5pt] (2,4)-- (10,0);
\draw [line width=0.5pt] (0,10)-- (8,6);
\draw [line width=0.5pt] (10,0)-- (8,6);

\draw [line width=2pt] (1.33333333333,6)-- (4,6);
\draw [line width=1.5pt] (4,6)-- (8,2);
\draw [line width=1.5pt] (4,6)-- (5.333333333,7.33333333);
\draw [line width=1.5pt] (8,2)-- (7.333333333,1.33333333);
\draw [line width=2pt] (8,2)-- (9.3333333333,2);

\begin{scriptsize}
\draw (2,7) node {$2$};
\draw (8.5,2.5) node {$2$};
\end{scriptsize}
\end{tikzpicture}
\\
$(a)$ & $(b)$\\
\end{tabular}

\caption{\label{figure example tropical curves}Two examples of tropical curves in tropical tori.}
\end{center}
\end{figure}

    \begin{expl}
        On Figure \ref{figure example tropical curves} we see two examples of tropical curves in their respective abelian surface. In each case, the basis of $\ZZ^2$ is the canonical basis, and the basis of $\Lambda$ is given by the bottom and left sides of the parallelogram.
            \begin{itemize}
                \item For the first tropical curve, the degree is $B=\left(\begin{smallmatrix}
                    2 & 0 \\ 0 & 3 \\
                \end{smallmatrix}\right)$, and the lattice is given by a matrix of the form $S=\left(\begin{smallmatrix}
                    \alpha & 2\gamma \\ 3\gamma & \beta \\
                \end{smallmatrix}\right)$, with $\alpha,\beta>0$ and $\det S>0$ for orientation purpose. The polarization is $C=\left(\begin{smallmatrix}
                    3 & 0 \\ 0 & 2 \\
                \end{smallmatrix}\right)$, and we may check that $BS^\intercal$ is positive definite.
                \item The second tropical curve gives an example where the matrices are not diagonal. In this situation, we have $B=\left(\begin{smallmatrix}
                    2 & 1 \\ 0 & 1 \\
                \end{smallmatrix}\right)$, the lattice is given by a matrix of the form $S=\left(\begin{smallmatrix}
                    \alpha & 2\gamma+\beta \\ \gamma & \beta \\
                \end{smallmatrix}\right)$. The polarization is $C=\left(\begin{smallmatrix}
                    1 & 0 \\ -1 & 2 \\
                \end{smallmatrix}\right)$, and we may check that $BS^\intercal$ is positive definite.
            \end{itemize}
    \end{expl}

    We finish with the definition of the Mikhalkin multiplicity of a tropical curve.

    \begin{defi}
        Let $h:\Gamma\to\TT A$ be a parametrized tropical curve with $\Gamma$ trivalent and $h$ an immersion (except for eventual contracted ends due to marked points). The Mikhalkin multiplicity of $\Gamma$ is $\prod m_V$, where the product is over trivalent vertices (not adjacent to a marked point) and $m_V$ is the index of the lattice spanned by adjacent slopes, or equivalently the determinant of two among the outgoing slopes.
    \end{defi}

\section{Deformation and Mumford families}

    \subsection{Mumford families}

    Let $D^\times\subset\CC$ be the punctured unit disk and $\Lambda$ a rank $2$ lattice with a preferred basis. Take two morphisms $S:\Lambda\to\RR^2$ given by a matrix $(s_{ij})$ and $Z:\Lambda\to\CC^2$ given by a matrix $(z_{ij})$. We construct the family of lattices in $\CC^2$ depending on a parameter $t\in D^\times$ given by the period matrix $(I\ Z_t)$, where
    $$Z_t=Z+S\frac{\log t}{2i\pi},$$
    where $\log t$ is a formal logarithm to the parameter $t$. To make it well-defined, assume that $S$ has values in $\QQ^2$. Write $S=\frac{S'}{r}$ with $S'$ having integer coefficients. Then, up to scaling, which amounts to a base change on $D^\times$ setting $t=(t')^r$, we can assume that $S$ has integer coefficients. Under the latter assumption, changing $\log t$ by some multiple $2i\pi$ does not change the lattice spanned by $(I\ Z_t)$.
    
    In the multiplicative setting, we set $\zeta_{ij}=e^{2i\pi z_{ij}}$ so that $(\zeta_{ij})$ gives the morphism $\Lambda\to(\CC^*)^2$. The complex torus is $(\CC^*)^2/\gen{(\zeta_{ij}t^{s_{ij}})}$.

    \begin{defi}
    Let $(Z,S)$ be chosen as above.
    \begin{itemize}
        \item The family of complex tori determined by the data $(Z,S)$ is called a Mumford family and is denoted by $\AAA(Z,S)$.
        \item The tropicalization of $\AAA(Z,S)$, denoted by $\TT\AAA(Z,S)$, is the tropical torus $\RR^2/\Lambda$ where $\Lambda\to\RR^2$ is the inclusion given by $S$.
    \end{itemize}
    \end{defi}

    By \textit{point} in $\AAA(Z,S)$, we mean a section of the form $t\mapsto (\xi_1 t^{u_1},\xi_2 t^{u_2})\in (\CC^*)^2/\Lambda_t$. Its tropicalization is the point $(u_1,u_2)\in\RR^2/\Lambda$.

    \subsection{Polarization on a Mumford family and tropicalization}

    We now give the definition of a polarization on a Mumford family, and study its relation with tropical polarizations.

    \begin{defi}
    Let $\AAA(Z,S)$ be a Mumford family. A matrix $Q$ is a polarization of the family if it induces a polarization on each torus of the family. The family is called a \textit{polarized Mumford family}.
    \end{defi}

    In \cite[Chapter 2.6]{griffiths2014principles}, given a polarization $Q$ on an abelian surface $\CC A$, it is possible to pick a basis where $Q$ is in diagonal anti-diagonal form: $Q=\left(\begin{smallmatrix}
    0 & \Delta \\ -\Delta & 0 \end{smallmatrix}\right)$, where $\Delta$ is a diagonal matrix with positive entries dividing each other. The entries of $\Delta$ are uniquely determined and form the \textit{type} of the polarization. In such a form, it is possible to deform $\CC A$ in a Mumford family where $Q$ keeps being a polarization. This is used in \cite{blomme2022abelian1} to apply Nishinou's correspondence theorem \cite{nishinou2020realization}. However, we do not actually need this anti-diagonal form to deform in a Mumford family keeping the polarization. The precise relation is as follows.

    \begin{prop}\label{prop-mumford-polarization}
        Let $\AAA(Z,S)$ be a Mumford family with $S$ an invertible real matrix. Polarizations on the family are the integer skew-symmetric matrices $Q$ satisfying the following conditions:
            \begin{enumerate}
                \item The matrix and its Poincar\'e dual have the following form:
                $$Q=\left(\begin{smallmatrix}
            0 & C \\ -C^\intercal & T
        \end{smallmatrix}\right) \text{ and }\operatorname{Pf}(Q)Q^{-1}=\left(\begin{smallmatrix}
            -T & B \\ -B^\intercal & 0
        \end{smallmatrix}\right),$$
                where $B$ and $C$ are the comatrix of each other,
                \item $BS^\intercal$ is symmetric and positive definite: $C$ is a tropical polarization of the tropicalization $\TT\AAA(Z,S)$,
                \item $Q$ satisfies Riemann bilinear relation for $\Omega=(I\ Z)$:
                    \begin{itemize}
                        \item $BZ^\intercal-ZB^\intercal=T$,
                        \item $-i\operatorname{Pf}(Q)(B\overline{Z}^\intercal-ZB^\intercal-T)$ is (hermitian and) positive definite.
                    \end{itemize}
            \end{enumerate}
        Conversely, given a polarization $Q$ on an abelian surface $\CC A$ determined by the period matrix $\Omega=(I\ Z)$ and having the form stated in (1), we can deform $\CC A$ in a Mumford family $\AAA(Z,S)$ by picking any $S$ such that $C$ induces a polarization of the associated tropical torus.
    \end{prop}

    \begin{proof}
        Assume that $\operatorname{Pf}(Q)Q^{-1}=\left(\begin{smallmatrix}
        -T & B \\ -B^\intercal & K \end{smallmatrix}\right)$. Then, the first part of Riemann-bilinear relation for the period matrix $(I\ Z_t)$ writes itself as follows:
        $$-T+BZ^\intercal-ZB^\intercal +ZKZ^\intercal +(BS^\intercal-SB^\intercal + ZKS^\intercal + SKZ^\intercal)\frac{\log t}{2i\pi} + SKS^\intercal\left(\frac{\log t}{2i\pi}\right)^2 =0 .$$
        As the relation is true for every $t$ and a polynomial in $\frac{\log t}{2i\pi}$, we deduce that all coefficients are $0$.
        \begin{enumerate}
            \item As $S$ is invertible, the vanishing of the last coefficient implies that $K=0$, so that $\operatorname{Pf}(Q)Q^{-1}$ has the desired form. The form of $Q$ follows from the form of $\operatorname{Pf}(Q)Q^{-1}$, proving (1).
            \item Using $K=0$, the vanishing of the second coefficient is now equivalent to $BS^\intercal$ being symmetric.
            \item The vanishing of the first coefficient yields the first part of (3).
        \end{enumerate}

        We now write the second Riemann-bilinear condition: $-i(I\ Z_t)Q^{-1}(I\ \overline{Z_t})^\intercal$ is positive definite. We multiply by $\operatorname{Pf}(Q)^2=-(\det B)\operatorname{Pf}(Q)$. Using that $BS^\intercal$ is symmetric, we get the following positive matrix:
        \begin{align*}
             & i(\det B)\left( -T+B\overline{Z}^\intercal - ZB^\intercal + BS^\intercal\left( -\frac{\log t}{2i\pi}-\frac{\overline{\log t}}{2i\pi}\right) \right) \\
            = & i(\det B)(-T+B\overline{Z}^\intercal - ZB^\intercal) +\frac{\det B}{\pi}(-\log|t|)BS^\intercal. \\
        \end{align*}
        Taking $|t|\to 0$, we deduce that $(\det B)BS^\intercal$ is positive definite. Since $\det S>0$, we also have $\det B>0$ and $BS^\intercal$ is positive definite, finishing the proof of (2). Taking $|t|\to 1$, we finish the proof of (3).
        
        Conversely, starting with a polarization $Q$ of the given form for $\Omega=(I\ Z)$, picking any matrix $S$ such that $C$ induces a tropical polarization on the corresponding tropical torus, it is clear we get a polarized Mumford family.
    \end{proof}

    Writing $B=\left(\begin{smallmatrix}
        b_{11} & b_{21} \\ b_{12} & b_{22}
    \end{smallmatrix}\right)$, $Z=\left(\begin{smallmatrix}
        z_{11} & z_{21} \\ z_{12} & z_{22}
    \end{smallmatrix}\right)$ and $T=\left(\begin{smallmatrix}
        0 & \tau \\ -\tau & 0
    \end{smallmatrix}\right)$, the first Riemann condition on $B,Z,T$ is expressed as follows:
    $$b_{11}z_{12}+b_{21}z_{22}-b_{12}z_{11}-b_{22}z_{21}=\tau.$$

    As the induced tropical polarization on $\TT \AAA(Z,S)$ only depends on $C$ and not $T$, the latter may not be primitive even if the starting polarization $Q$ is primitive. On the curve side, it means the tropical degree $B$ may be divisible even though the complex curve class $\operatorname{Pf}(Q)Q^{-1}$ is not.
    
    Maybe more concretely, the class of a complex curve is $\beta=\operatorname{Pf}(Q)Q^{-1}$, Poincar\'e dual to the cher class. The degree of the trpicalization is only its upper-right term $B$. The classes $\operatorname{Pf}(Q)Q^{-1}$ and $\left(\begin{smallmatrix} 0 & B \\ -B^\intercal & 0 \end{smallmatrix}\right)$ may differ by a multiple of the fiber $(S^1)^2\hookrightarrow (\CC^*)^2/\Lambda$ of the tropicalization map, precisely this $\left(\begin{smallmatrix} -T & 0 \\ 0 & 0 \end{smallmatrix}\right)$.

    \section{Correspondence Theorem}

    \subsection{Enumerative problems}

    \subsubsection{Complex enumerative geometry} Let $\CC A$ be a complex abelian surface with $\beta\in H_2(\CC A,\ZZ)$ a curve class given by the polarization. We can consider the moduli space of genus $g$ stable maps with $n$ marked points realizing the homology class $\beta$, denoted by $\M_{g,n}(\CC A,\beta)$. It is endowed with a reduced virtual fundamental class $[\M_{g,n}(\CC A,\beta)]^\mathrm{red}$ of rank $g+n$. We have the evaluation map
    $$\ev : \M_{g,n}(\CC A,\beta)\longrightarrow (\CC A)^n.$$
    We get a reduced Gromov-Witten invariant by taking $n=g$ and pulling-back the product of cohomology classes Poincaré dual to a point in $\CC A$:
    $$N_{g,\beta} = \int_{[\M_{g,n}(\CC A,\beta)]^\mathrm{red}} \prod_1^n\ev_i^*(\pt).$$

    The number $N_{g,\beta}$ is deformation invariant in families of polarized abelian surfaces, and only depends on $\beta$ through its divisibility $d$ and self-intersection $\beta^2=2d^2n$. We thus denote the invariant $N_{g,d,n}$. It corresponds to the number of genus $g$ curves in the class $\beta$ passing through a generic configuration of $g$ points in $\CC A$.

    \subsubsection{Tropical enumerative geometry} The tropical point of view on enumerative geometry in tropical abelian surfaces is studied in \cite{blomme2022abelian1}. All statements recalled in this section are proved in the latter. Although the tropical invariance may be proven by tropical methods, it can also be deduced from Nishinou's correspondence theorem \cite{nishinou2020realization}.
    
    Let $\TT A$ be a tropical abelian surface with a curve class $B$ given by its tropical polarization. We have a moduli space of genus $g$ parametrized tropical curves realizing the class $B$, denoted by $\M_{g,n}(\TT A,B)$. We have the evaluation map
    $$\ev:\M_{g,n}(\TT A,B)\longrightarrow (\TT A)^n.$$
    Taking $g=n$ and given a generic configuration $\P$ of $g$ points in $\TT A$, there is a finite number of tropical curves passing through $\P$.
    
    If $\TT A$ is furthermore assumed to be generic among the tropical abelian surfaces with a given polarization, and $\P$ is still a generic point configuration in $\TT A$, these tropical curves are trivalent and immersed (no flat vertex). We can thus count the tropical curves passing through $\P$ multiplicity $\delta_\Gamma\prod m_V$. The count does not depend on $\P$ in $\TT A$ nor on $\TT A$. Furthermore, if we only count curves with a given gcd $k|B$, we also get an invariant count, denoted by $N_{g,B,k}^\mathrm{trop}$.

    The multiplicity $m_\Gamma=\delta_\Gamma\prod m_V$ is $(4g-3)$-homogeneous in the sense that if we multiply every edge weight by $k$, the multiplicity is multiplied by $k^{4g-3}$ (a $k$ for the gcd and a $k^2$ for each vertex). Therefore, we have that $N^\mathrm{trop}_{g,B,k}=k^{4g-3}N^\mathrm{trop}_{g,B/k,1}$. Furthermore, through a change of basis, the invariant only depends on $B$ through its divisibility $d$ and its determinant $d^2n$.

    \begin{rem}
    More precisely, the multiplicity provided by the correspondence theorem from \cite{nishinou2020realization} is $\prod_e w_e$ times some lattice index, and we can see this multiplicity is $(4g-3)$-homogeneous, since we have $4g-3$ edges on a trivalent genus $g$ graph with $g$ marked points. Meanwhile, such a graph having $2g-2$, Mikhakin's multiplicity is only $(4g-4)$-homogeneous (a $k^2$ for each vertex), so we see that both cannot agree and we are missing a term: the gcd $\delta_\Gamma$.
    \end{rem}

    \subsection{Correspondence theorem}

    Choose a Mumford family $\AAA(Z,S)$ with tropicalization $\TT A$, which we assume to be polarized and generic in the space of tropical abelian surfaces with a given polarization. Choose a generic point configuration $\P$ in $\TT A$. We have a finite number of genus $g$ tropical curves of degree $B$ passing through $\P$.

    Assume $B=\left(\begin{smallmatrix} b_{11} & b_{21} \\ b_{12} & b_{22} \end{smallmatrix}\right)$, and for $\delta|B$, set
    $$\sigma(Z,B,\delta)= \zeta_{11}^{-b_{12}/\delta}\zeta_{12}^{b_{11}/\delta}\zeta_{21}^{-b_{22}/\delta}\zeta_{22}^{b_{21}/\delta}.$$
    We can now state the realization and correspondence theorem, coupled to the multiplicity computation from \cite{blomme2022abelian1}.

    \begin{theo}\cite[Theorem 1.1 \& 1.2]{nishinou2020realization}\label{theo-corresp}
        Let $\AAA(Z,S)$ be a Mumford family with generic tropicalization $\TT A$. Let $\P_t$ be a configuration of $g$ points with generic tropicalization $\P$. Let $h:\Gamma\to\TT A$ be a parametrized tropical curve passing through $\P$ with gcd $\delta_\Gamma$.
        \begin{enumerate}
            \item The parametrized tropical curve is realizable in $\AAA(Z,S)$ if and only if $\sigma(Z,B,\delta_\Gamma)=1$.
            \item If the curve is realizable, there are precisely $\delta_\Gamma\prod m_V$ complex curves passing through $\P_t$ and tropicalizing to $h:\Gamma\to\TT A$.
        \end{enumerate}
    \end{theo}

    The condition $\sigma(Z,B,\delta_\Gamma)=1$, referred later as \textit{realizability condition}, amounts to the Riemann-bilinear relation. Its necessity can be deduced from the Menelaus condition \cite{mikhalkin2017quantum}. In \cite{nishinou2020realization}, the right hand-side of the condition is rather $(-1)^{\sum m_V/\delta}$, but it is proved in \cite{blomme2022abelian1} that this quantity is actually always $1$. 

    \begin{rem}
        The union of images of the tropical curves induces a polyhedral subdivision of $\TT A$. The latter can be used to add a central fiber to $\AAA(Z,S)$ which the gluing of toric surfaces among their toric divisors (see \cite{nishinou2020realization} and \cite{blomme2022abelian1}). Nishinou's correspondence theorem states that the deformation of curves in the central fiber is essentially unobstructed.
    \end{rem}

    \section{Multiple cover formula}

    We now get to the proof of the multiple cover formula. The idea is to use two distinct Mumford families $\AAA(Z_0,S)$ and $\AAA(Z_1,S)$ carrying distinct polarizations $Q_0$ and $Q_1$, that yet induce the same tropical polarization $C$ on their common tropicalization. Applying the correspondence theorem in both families yields the curve counts for a primitive class, and for the non-primitive curve class. To conclude, we observe that the multiple cover formula is true for each tropical curve, which are the same for both Mumford families.

    Before getting to the proof, we describe the Mumford families we use.
    \begin{itemize}
        \item Let $B=\left(\begin{smallmatrix} b_{11} & b_{21} \\ b_{12} & b_{22} \end{smallmatrix}\right)$ be a matrix with divisibility $d$ and determinant $d^2n$. We could take an explicit $B$ but do not even need this level of specificity.
        \item Let $S$ be a (integer) matrix such that $B$ induces a tropical polarization on the corresponding abelian surface and assume that $S$ is generic among them. In particular, for a generic choice of point configuration $\P\in\TT A$, the genus $g$ degree $B$ curves passing through $\P$ are trivalent.
        \item Let $\tau\in\ZZ$ be some integer and let $T=\left(\begin{smallmatrix}
            0 & \tau \\
            -\tau & 0 \\
        \end{smallmatrix}\right)$.
    \end{itemize}

    Let $Z_\tau=\left(\begin{smallmatrix} z_{11}^\tau & z_{21}^\tau \\ z_{12}^\tau & z_{22}^\tau \end{smallmatrix}\right)$ be a complex matrix such that the complex torus $\CC^2/(I\ Z_\tau)$ admits the polarization $Q_\tau$ defined by
        $$\operatorname{Pf}(Q_\tau)Q_\tau^{-1}=\begin{pmatrix}
            -T & B \\
            -B^\intercal & 0
        \end{pmatrix}.$$
        Then, using Proposition \ref{prop-mumford-polarization}, the Mumford family $\AAA(Z_\tau,S)$ admits the matrix $Q_\tau$ as polarization. We set $\zeta_{ij}^\tau=e^{2i\pi z_{ij}^\tau}$ to be the multiplicative coordinates of the Mumford family. The realizability condition from Theorem \ref{theo-corresp} can be reformulated as follows.
        
        \begin{lem}
            The realizability condition in the Mumford family $\AAA(Z_\tau,S)$ for a trivalent immersed parametrized tropical $h:\Gamma\to\TT A$ is equivalent to $\delta_\Gamma$ dividing $\tau$.
        \end{lem}

        \begin{proof}        
        The first part of the Riemann bilinear relation coming from the polarization of the family $\AAA(Z_\tau,S)$ is $BZ_\tau^\intercal-Z_\tau B^\intercal = T$, which are size $2$ skew-symmetric matrices. Taking the upper-right coefficient, it can be rewritten
        $$b_{11}z_{12}^\tau+b_{21}z_{22}^\tau-b_{12}z_{11}^\tau-b_{22}z_{21}^\tau=\tau.$$
        Dividing by $\delta$ for some $\delta|d$ and taking the exponential, we get that
        $$(\zeta_{12}^\tau)^{b_{11}/\delta}(\zeta_{11}^\tau)^{-b_{12}/\delta}(\zeta_{22}^\tau)^{b_{21}/\delta}(\zeta_{21}^\tau)^{-b_{22}/\delta}=e^{2i\pi\frac{\tau}{\delta}},$$
        so that realizability condition $\sigma(Z_\tau,B,\delta)=1$ from Theorem \ref{theo-corresp} is equivalent to $e^{2i\pi\tau/d}=1$, or in other terms $\delta_\Gamma|\tau$.
        \end{proof}

    We can now prove the multiple cover formula.

    \begin{theo}\label{theo-multiple-cover}
        The invariants $N_{g,d,n}$ satisfy the multiple cover formula
        $$N_{g,d,n} = \sum_{k|d}k^{4g-3}N_{g,1,\left(\frac{d}{k}\right)^2n}.$$
    \end{theo}

    \begin{proof}
        We consider the values $\tau=0$ and $\tau=1$ and use the correspondence theorem for the Mumford families $\AAA(Z_0,S)$ and $\AAA(Z_1,S)$. Choose a generic point configuration $\P$ in $S$, with lifts $\P_0$ (resp. $\P_1$) in $\AAA(Z_0,S)$ (resp. $\AAA(Z_1,S)$).
        
        Let $\CCC(B)$ be the set of genus $g$ tropical curves of degree $B$ passing through $\P$. For $k|B$, let $\CCC_k(B)\subset\CCC(B)$ be the subset of curves with gcd $k$. Stress that we have a bijection between $\CCC_k(B)$ and $\CCC_1(B/k)$ obtained by just dilating edge weights (and lengths).
        
        As the tropicalizations of $\AAA(Z_0,S)$ and $\AAA(Z_1,S)$ are the same, the tropical solutions are the same.

        \begin{itemize}
            \item For the family $(Z_0,S)$, the polarization $Q_0$ has divisibility $d$ and self-intersection $2d^2n$. The phase condition $\delta_\Gamma|0$ is always satisfied. Therefore, we get
            $$N_{g,d,n} = \sum_{\Gamma\in\CCC(B)} m_\Gamma.$$

            \item For the family $(Z_1,S)$, the polarization $Q_1$ is primitive and also has self-intersection $2d^2n$. The phase condition $\delta_\Gamma|1$ is only satisfied for curves with gcd $1$, so-called \textit{primitive curves}. Therefore, we get that
            $$N_{g,1,d^2n} = \sum_{\Gamma\in\CCC_1(B)}m_\Gamma \  \left(= N^\mathrm{trop}_{g,B,1}\right).$$
        \end{itemize}

        To conclude, we only need to sort out tropical curves in the first point according to their gcd, and apply the second point for the classes $B/k$:
        \begin{align*}
            N_{g,d,n} = \sum_{\Gamma\in\CCC(B)} m_\Gamma =& \sum_{k|d} \sum_{\Gamma\in\CCC_k(B)} m_\Gamma \\
            =& \sum_{k|d} \sum_{\widetilde{\Gamma}\in\CCC_1(B/k)} m_{k\widetilde{\Gamma}} \\
            =& \sum_{k|d} k^{4g-3}\sum_{\widetilde{\Gamma}\in\CCC_1(B/k)} m_{\widetilde{\Gamma}} \\
            =& \sum_{k|d} N_{g,1,\left(\frac{d}{k}\right)^2n}.\\
        \end{align*}
    \end{proof}

\bibliographystyle{alpha}
\bibliography{biblio}

\end{document}